\documentclass[letterpaper]{article}
\usepackage{graphicx, amsmath, amsthm, amssymb, tikz, microtype}
\usepackage[letterpaper]{geometry}
\usepackage{enumitem}

\usepackage[draft=false,pdftex]{hyperref}
\hypersetup{
    breaklinks=true,   
    bookmarks=true,         
    unicode=false,          
    pdftoolbar=true,        
    pdfmenubar=true,        
    pdffitwindow=false,     
    pdfstartview={FitH},    
    pdftitle={Regular matchstick graphs on the sphere},    
    pdfauthor={Konrad J. Swanepoel},     
    pdfauthor={},     
    pdfnewwindow=false,      
    colorlinks=true,       
    linkcolor=black,          
    citecolor=black,        
    filecolor=black,      
    urlcolor=black           
}

\newcommand{\area}{\operatorname{area}}
\newtheorem*{theorem}{Theorem}

\title{Regular matchstick graphs on the sphere\footnote{Research partially supported by ERC grant no. 882971, ``GeoScape,'' and by the Erd\H os Center.}
}
\author{Konrad J. Swanepoel\\ Department of Mathematics, London School of Economics and Political Science}
\date{}

\begin{document}
\maketitle
\begin{abstract}
We show that the $5$-regular matchstick graphs on the sphere are exactly the five $5$-regular contact graphs of congruent caps on the sphere found by R. M. Robinson (1969).

\medskip
    \centering
    \includegraphics[scale=0.0605, trim = 0 1 0 1, clip]{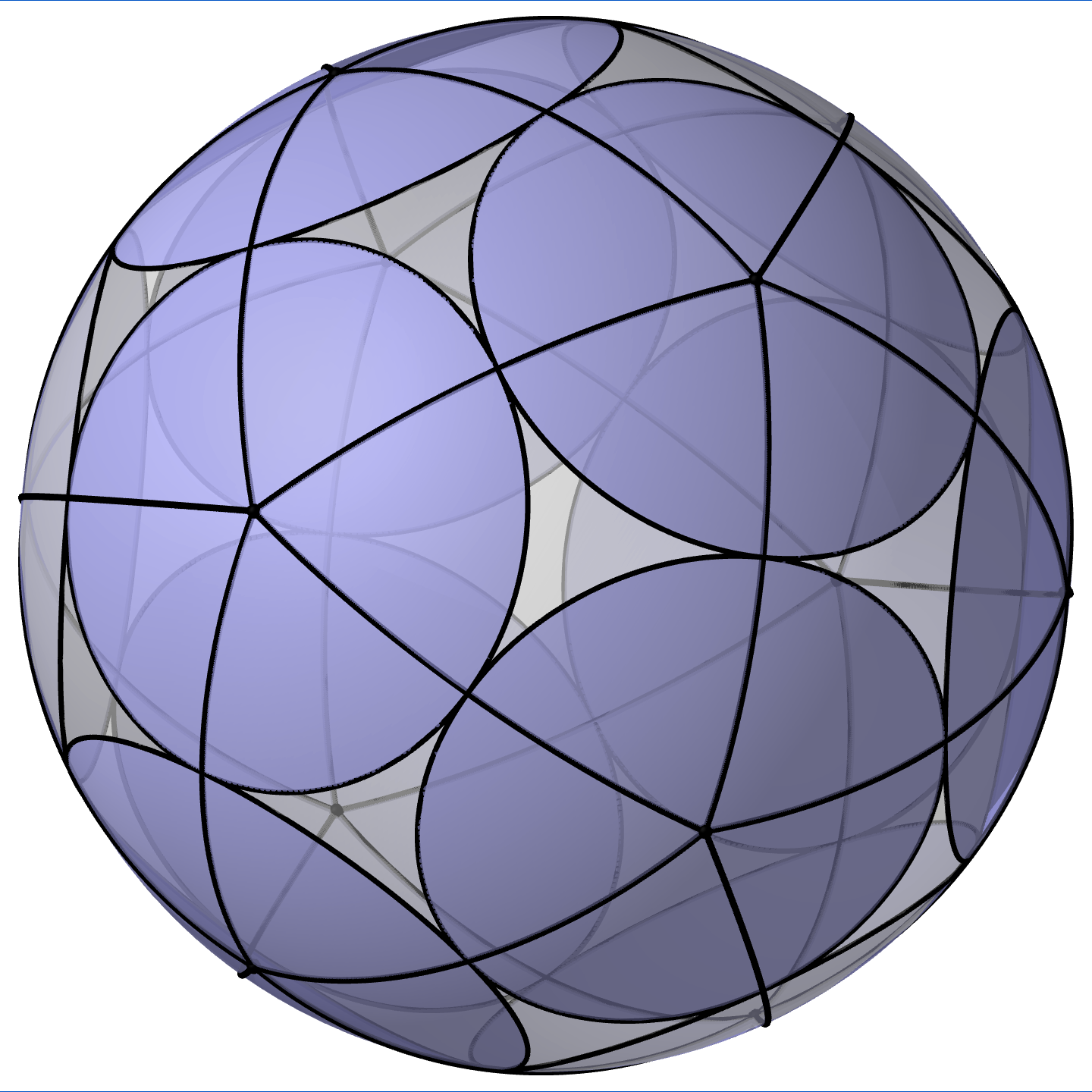} \quad
    \includegraphics[scale=0.0605, trim = 0 1 0 1, clip]{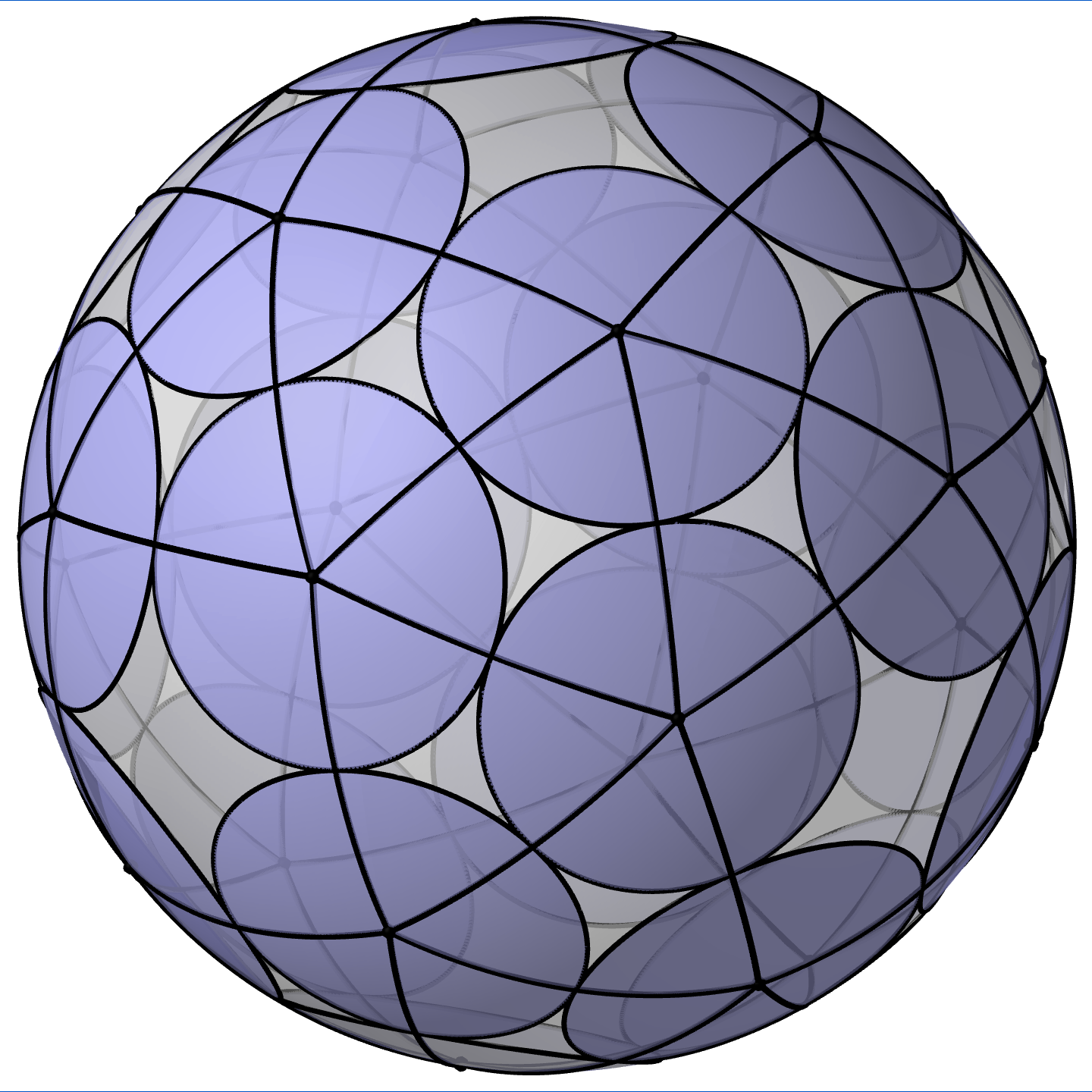} \quad
    \includegraphics[scale=0.0605, trim = 0 1 0 1, clip]{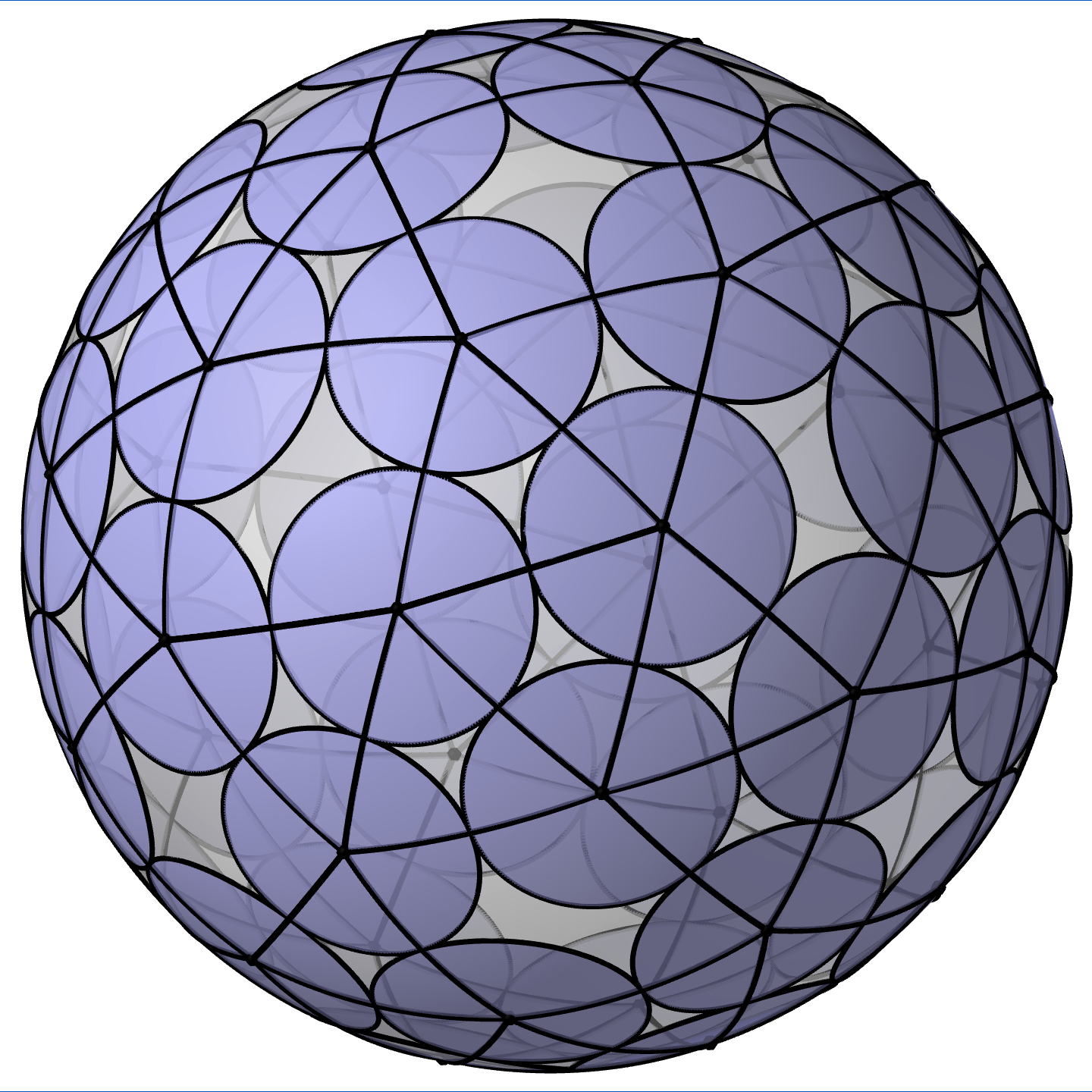}

    \includegraphics[scale=0.0605, trim = 0 1 0 1, clip]{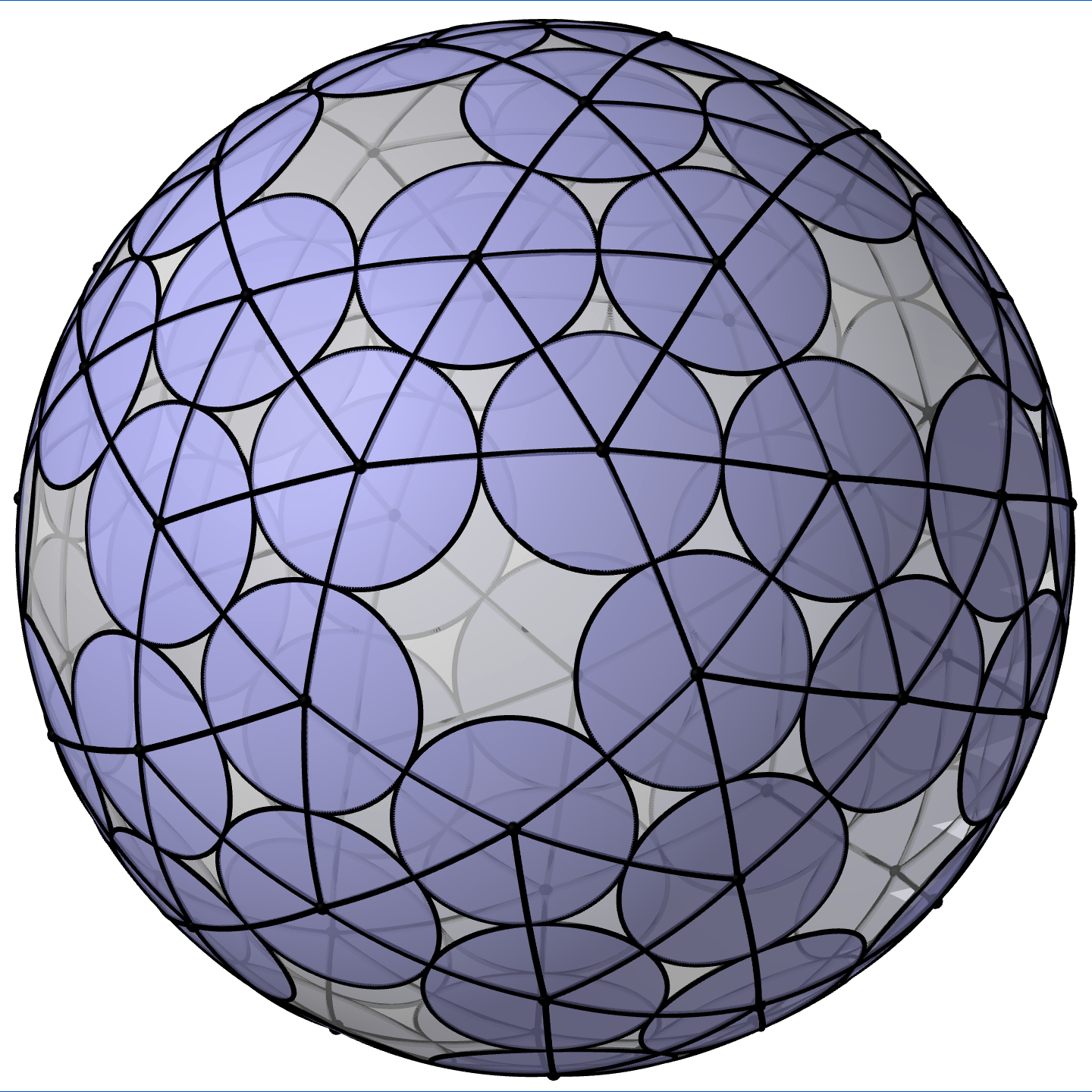} \quad
    \includegraphics[scale=0.0605, trim = 0 1 0 1, clip]{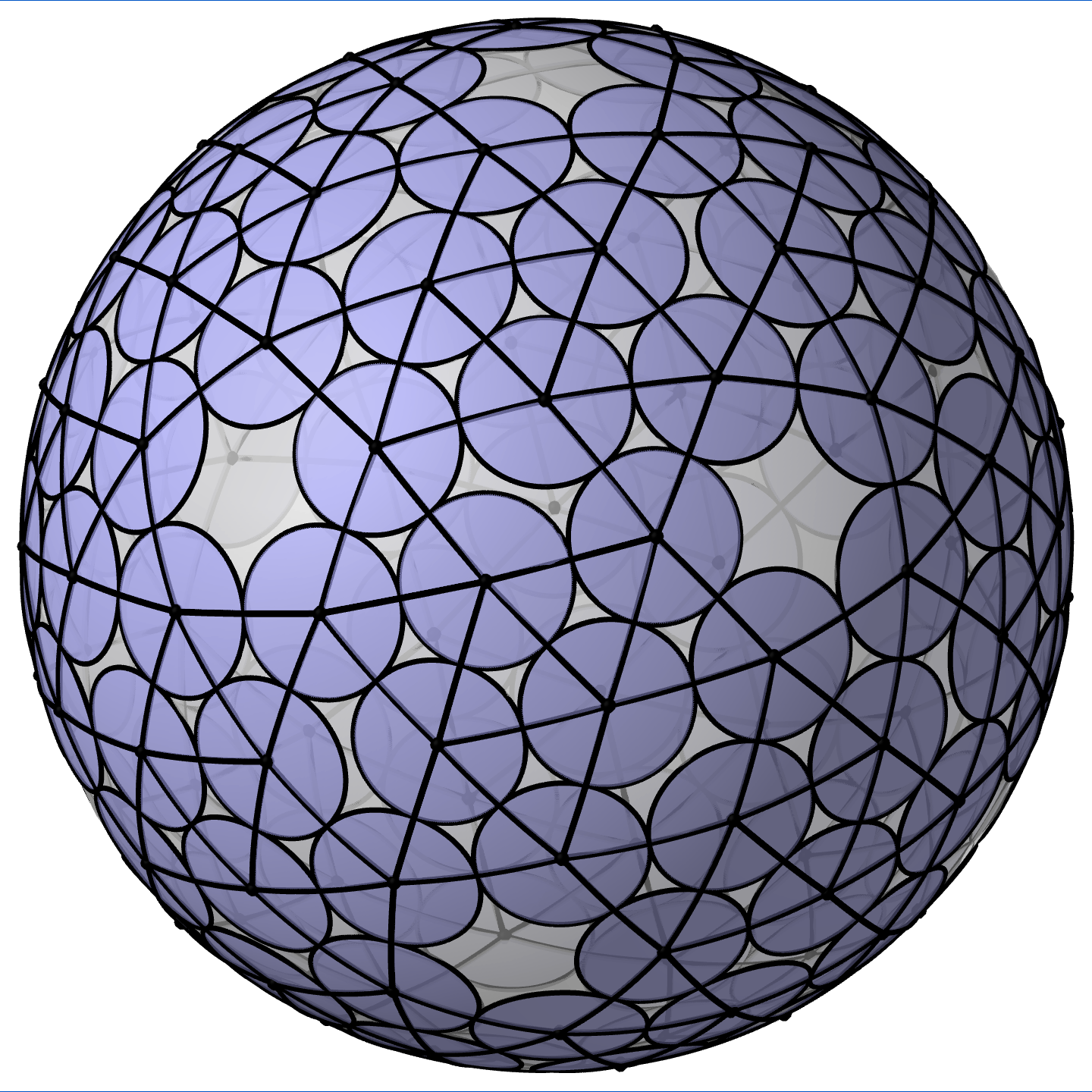}
\end{abstract}

\section{Introduction}
A \emph{matchstick graph} is a graph drawn in the plane with each edge represented as a unit straight-line segment, such that two edges have a common point only if this point is a common endpoint of the two edges.
Harborth \cite{oberwolfach,lighter} introduced this notion, and asked for which values of $k$ there exists a $k$-regular matchstick graph in the plane.
It is easy to find the smallest such graphs for $k\leq 3$.
Harborth's example \cite{lighter} of a $4$-regular matchstick graph with $52$ vertices is still the smallest known one, while the currently best lower bound is $34$ vertices, shown by Kurz \cite{kurz-4-regular}.
Blokhuis \cite{blokhuis} showed that there are no $5$-regular matchstick graphs in the Euclidean plane.
Kurz and Pinchasi in this \textsc{Monthly} \cite{kurz} gave a ``book proof'' (see \cite{klarreich}) that uses the Euler polyhedral formula combined with a charging argument.
As they observed at the end of their paper, following their argument on the sphere does not lead to a contradiction, and indeed the icosahedral net on the sphere is a $5$-regular matchstick graph (Figure~\ref{fig12}).

Here, we define a \emph{matchstick graph on the sphere} to be a simple graph drawn on the sphere with each edge represented as a minor great circular arc of some fixed angular length $\lambda < \pi$, in such a way that there are no crossing edges.
We will show that a modification of the charging argument in \cite{kurz} gives that any $5$-regular matchstick graph on the sphere is the contact graph of a packing of equal circular caps, each touching exactly $5$ others.
The \emph{contact graph} of a packing of spherical caps is the graph with vertex set the centers of the caps, and with edges the great circular arcs between the centers of any two touching caps.

The $5$-regular contact graphs of packings of equal spherical caps have already been classified in 1969 by R. M. Robinson \cite{Robinson1969}; there are five of them.
Three of them are the vertex sets of well-known polytopes: the icosahedron ($12$ vertices, Figure~\ref{fig12}), the snub cube ($24$ vertices, Figure~\ref{fig24}), and the snub icosahedron ($60$ vertices, Figure~\ref{fig60}).
Two further ones, on $48$ and $120$ vertices (Figures~\ref{fig48} and \ref{fig120}) both come from the union of two orbits of the symmetry group of the cube and of icosahedron, respectively.
These five packings are depicted, together with their contact graphs, in Figures~\ref{fig12} to \ref{fig120}.
While a drawing of the $5$-regular packing of $48$ congruent spherical caps exists in the literature \cite[Fig.~11]{Clare-Kepert}, we are not aware of any previous drawing of the packing of $120$ caps.
We drew these figures using SageMath, and html versions of them that can be rotated are available as ancillary files to this arXiv submission.

Thus, we obtain a complete classification of all $5$-regular matchstick graphs on the sphere.
The proof even allows us to obtain a slightly stronger result, where we only assume that each vertex has at least $5$ neighbors.

\begin{figure}
    \includegraphics[scale=0.138, trim = 0 1 0 1, clip]{12.png} \hfill  \includegraphics[scale=0.145]{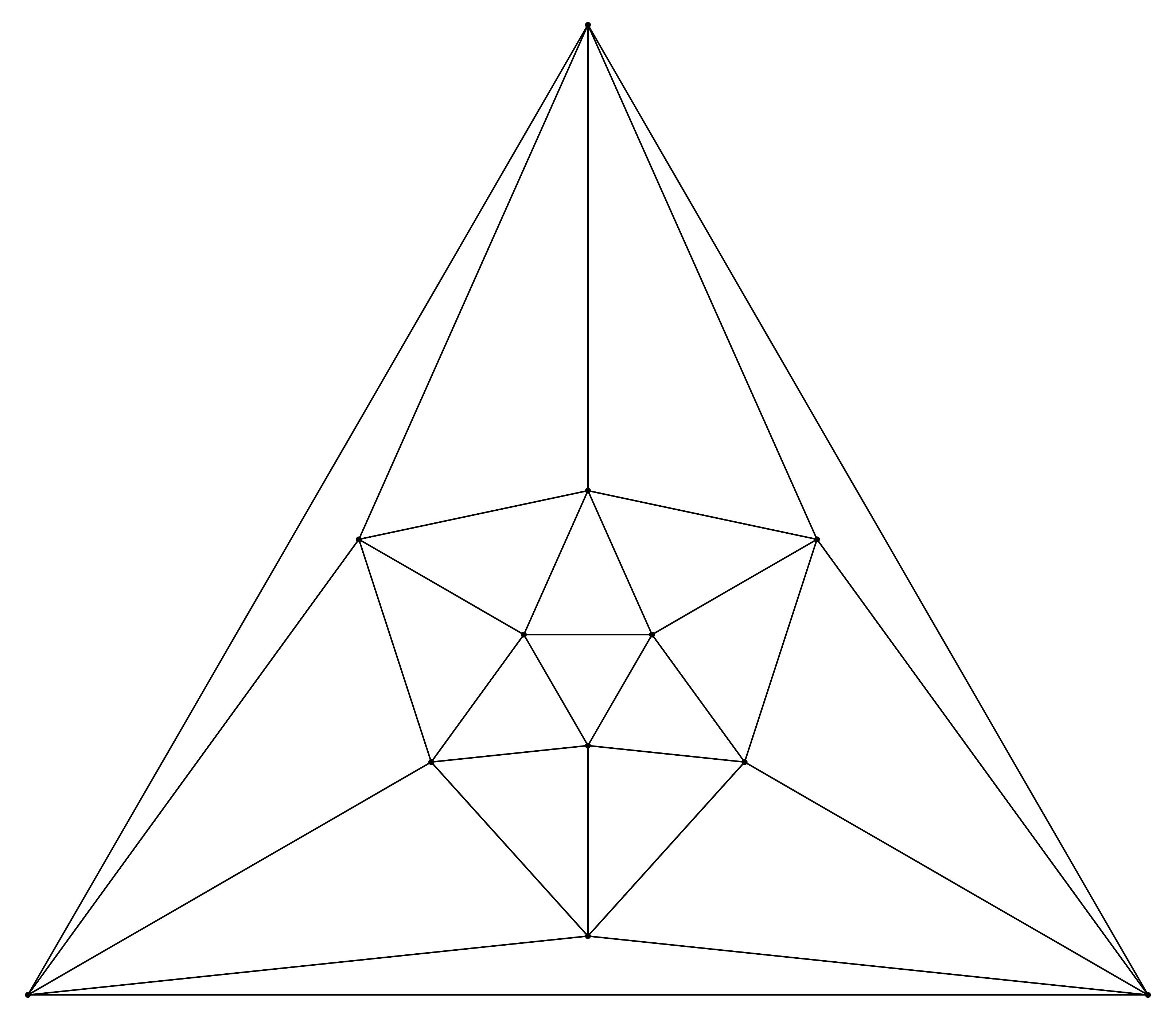}
    \caption{Icosahedron gives a $5$-regular matchstick graph on $12$ vertices on the sphere}\label{fig12}
\end{figure}
\begin{figure}
    \includegraphics[scale=0.138, trim = 0 1 0 1, clip]{24.png} \hfill  \includegraphics[scale=0.145]{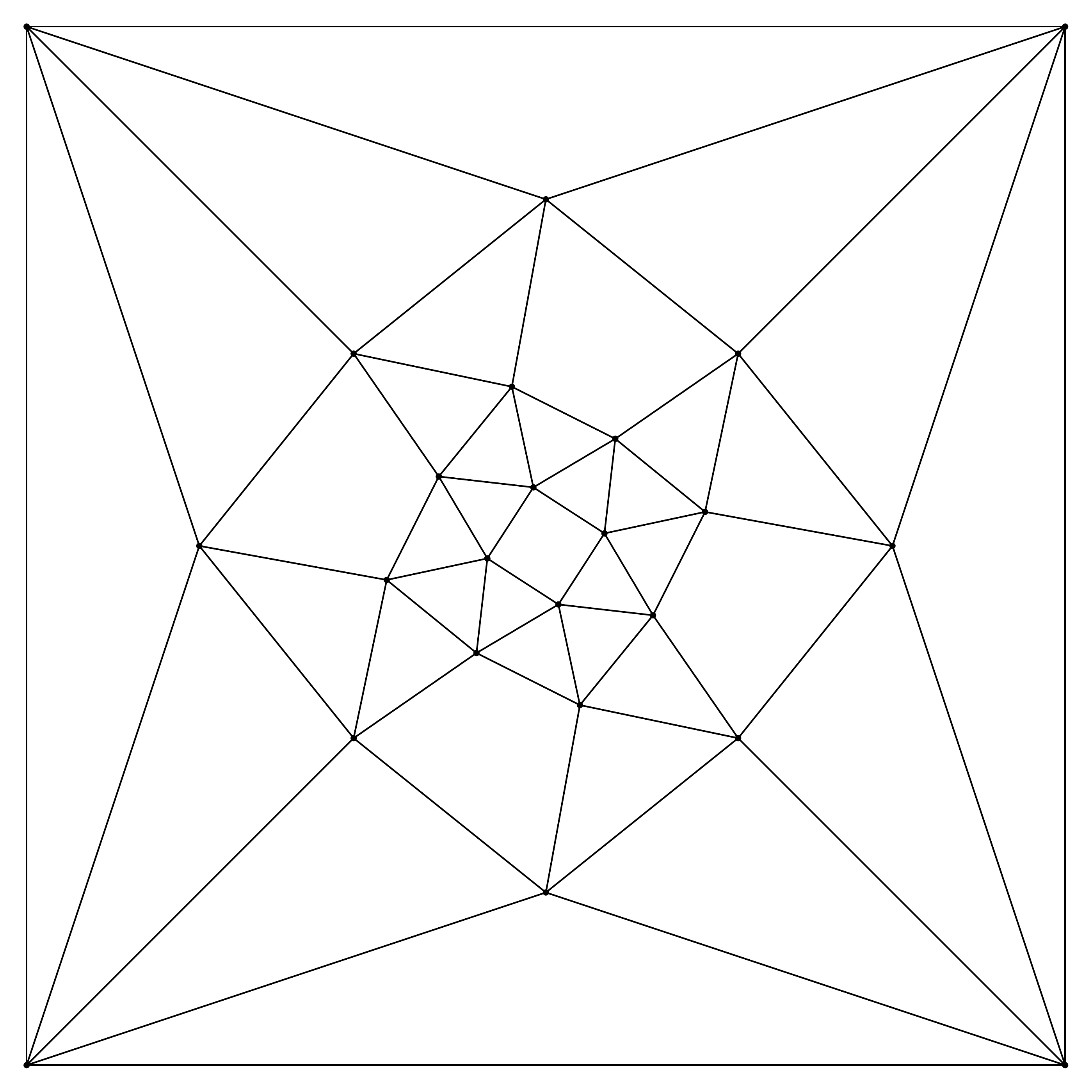}
    \caption{Snub cube gives a $5$-regular matchstick graph on $24$ vertices on the sphere}\label{fig24}
\end{figure}
\begin{figure}
    \includegraphics[scale=0.138, trim = 0 1 0 1, clip]{48.png} \hfill  \includegraphics[scale=0.145]{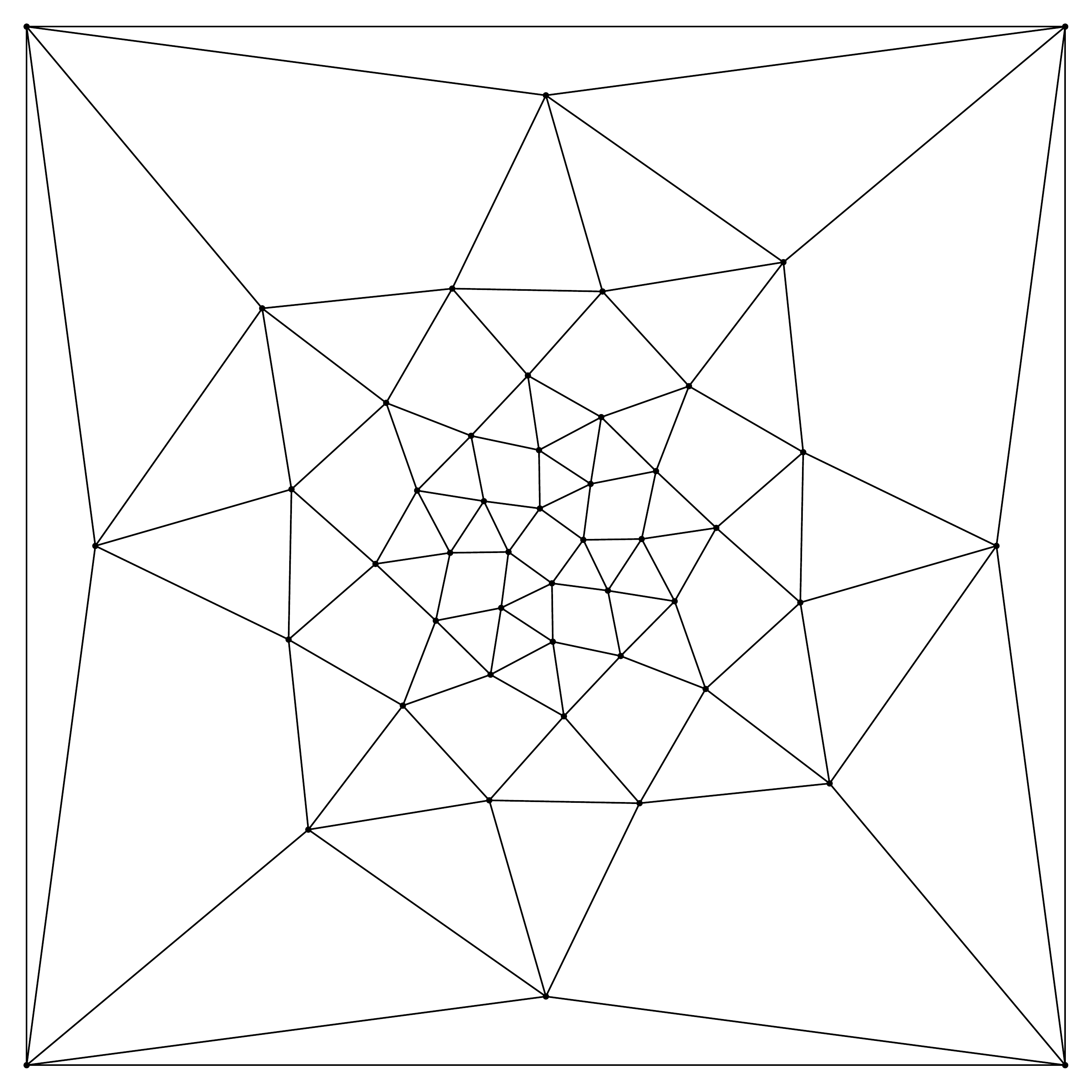}
    \caption{Robinson's $5$-regular matchstick graph on $48$ vertices on the sphere}\label{fig48}
\end{figure}
\begin{figure}
    \includegraphics[scale=0.142, trim = 0 1 0 1, clip]{60.png} \hfill  \includegraphics[scale=0.155]{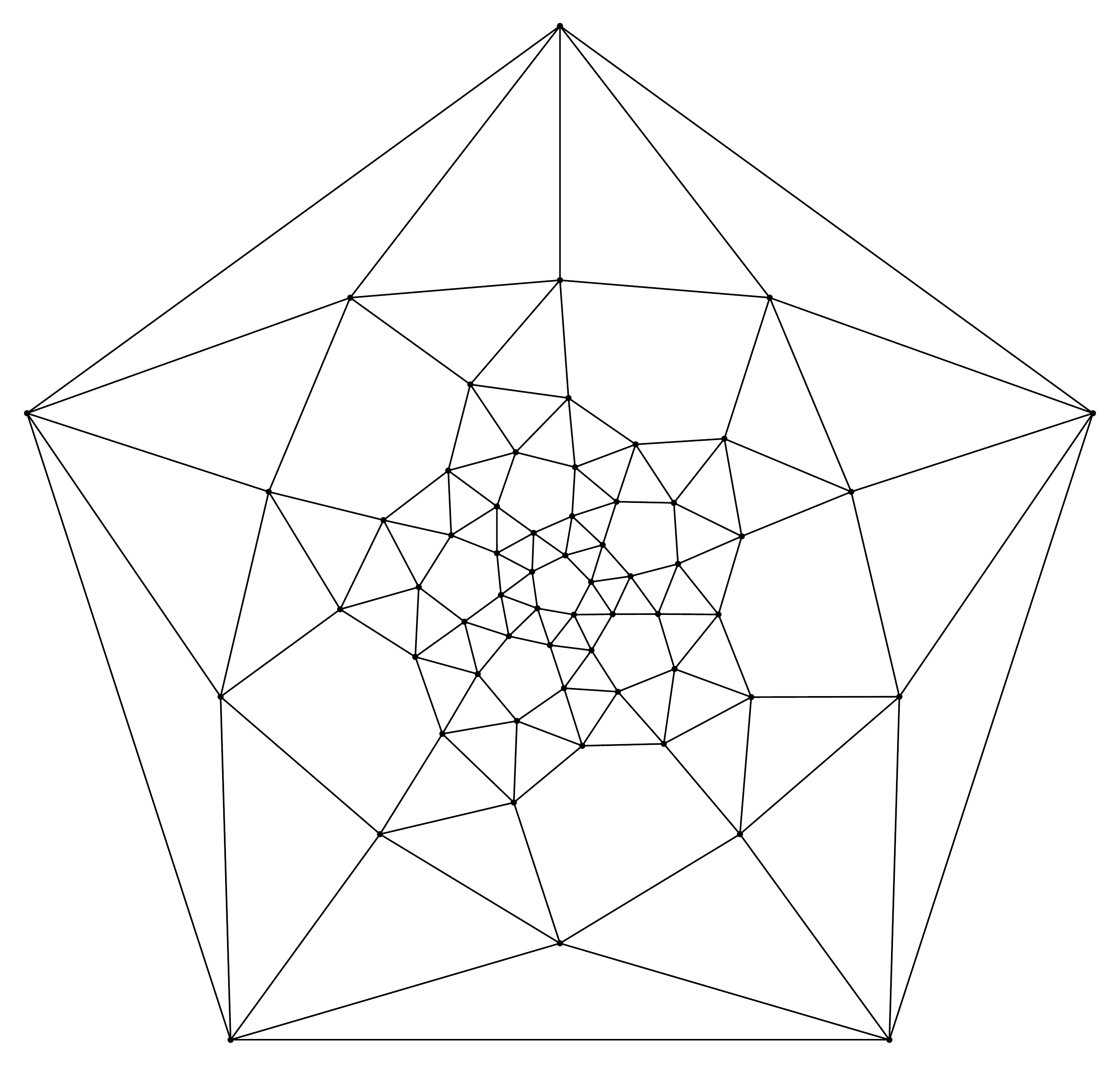}
    \caption{Snub icosahedron gives a $5$-regular matchstick graph on $60$ vertices on the sphere}\label{fig60}
\end{figure}
\begin{figure}
     \includegraphics[scale=0.142, trim = 0 1 0 1, clip]{120.png} \hfill  \includegraphics[scale=0.155]{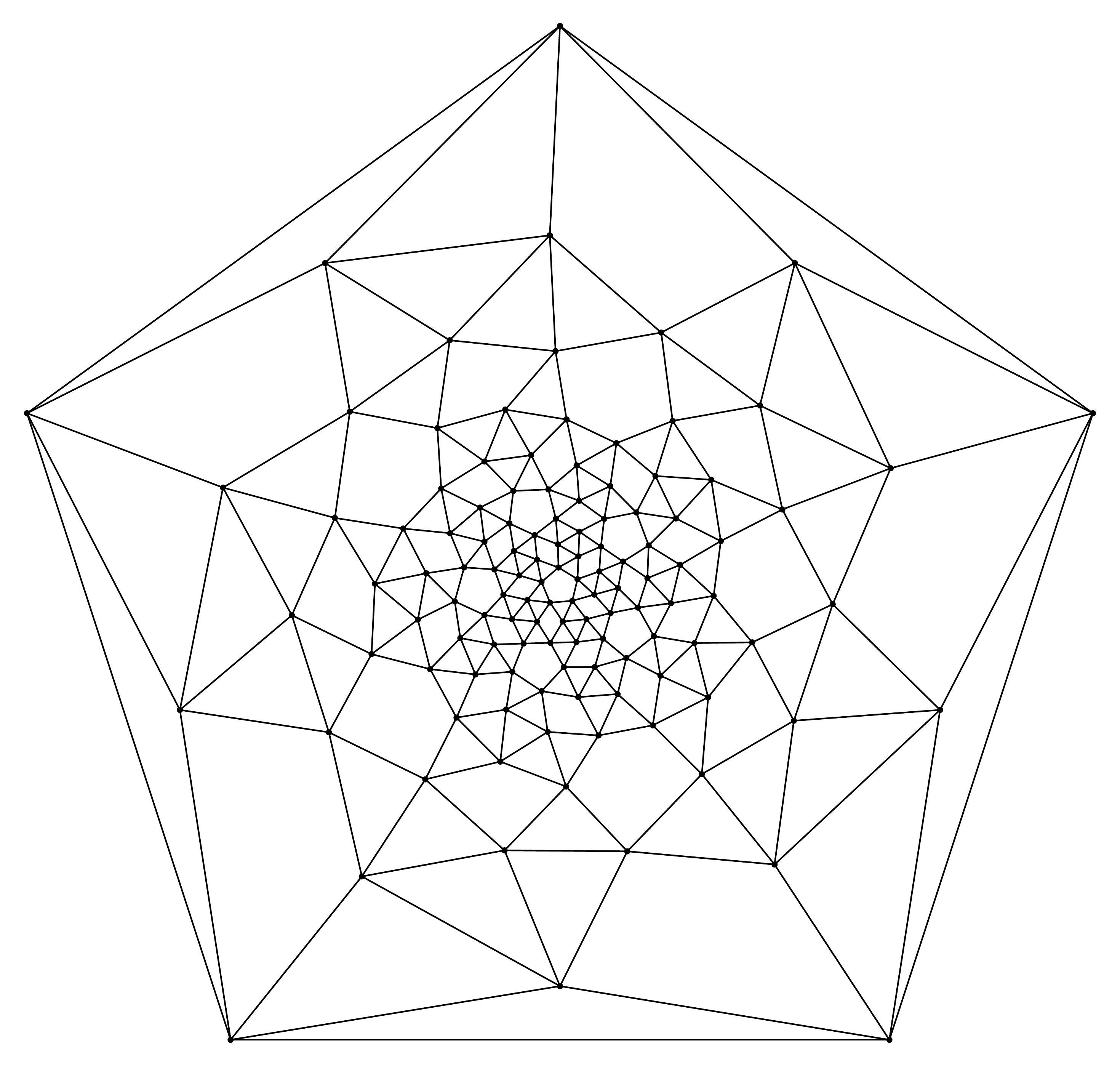}
    \caption{Robinson's $5$-regular matchstick graph on $120$ vertices on the sphere}\label{fig120}
\end{figure}

The assumption that the length $\lambda < \pi$ is not much of a restriction, because if $\lambda \geq \pi$, then no vertex of a matchstick graph with edges of angular length $\lambda$ can be incident to more than one edge without edges crossing, unless $\lambda=\pi$ and the vertex is joined to its diametrically opposite point with more than one edge (but then $G$ is not a simple graph).
If $\lambda < \pi$, then $G$ is necessarily a simple graph, which is what we assume throughout the paper.
Since we will only consider angular lengths, the radius of the sphere is irrelevant.
However, we will be working with the area of regions on the sphere, so we fix the sphere to be the unit sphere $\mathbb{S}^2$ of area $4\pi$.

\smallskip
In the proof of our Theorem below, we will use some elementary results from spherical geometry:

\begin{enumerate}[wide, labelindent=0pt]
\item The area of a simple spherical polygon with $n$ sides consisting of great circular arcs is equal to $\sum_{i=1}^k\alpha_i - (n-2)\pi$, where $\alpha_1,\dots,\alpha_n$ are the interior angles at the vertices of the polygon.
\item If two sides of a spherical triangle have equal angular length (in other words, the spherical triangle is isosceles) then the opposite angles are equal.
Also, in a spherical triangle with two sides $a$ and $b$ and corresponding opposite angles $\alpha$ and $\beta$, if $a > b$, then $\alpha > \beta$.
\item The spherical triangle inequality: $a < b+c$ in a spherical triangle with sides $a, b, c$.
\end{enumerate}

\begin{theorem}\label{maintheorem}
There are exactly five matchstick graphs on the sphere with minimum degree $5.$
They are Robinson's five $5$-regular contact graphs of packings of congruent spherical caps \textup{\cite{Robinson1969}} \textup{(}Figures~\ref{fig12} to \ref{fig120}\textup{)}.
\end{theorem}

\begin{proof}
By Robinson's result, it is sufficient to prove that any matchstick graph $G$ on the sphere with minimum degree $5$ is necessarily a $5$-regular contact graph of a packing of equal spherical caps.

Let $V$ denote the vertex set and $E$ the edge set of $G$.
We denote by $\|G\|$ the set of points in the union of all the minor arcs, including their endpoints, in the drawing of $G$ (also called the \emph{carrier} of $G$).
We define a \emph{face} of $G$ as a connected component of the complement $\mathbb{S}^2\setminus \|G\|$ of $\|G\|$ in the sphere.
Let $F$ denote the set of faces of $G$.
By the Euler formula for plane graphs, $|V|-|E|+|F|\geq 2$, with equality only if $G$ is connected.
An edge of $G$ \emph{belongs to} a face if there are points in the face arbitrarily close to some point in the relative interior of the edge.
Similarly, a vertex of $G$ \emph{belongs to} a face if there are points in the face arbitrarily close to the vertex.
The \emph{boundary} of a face is the subgraph consisting of all the vertices and edges belonging to it.
We define the \emph{degree of a face} $f$, denoted by $\deg(f)$, to be the number of edges belonging to $f$, 
where an edge belonging to $f$ is counted twice if $f$ lies on both sides of the edge.
Then $\sum_{f\in F} \deg(f) = 2|E|$.
As usual, the \emph{degree of a vertex} $v$, denoted by $\deg(v)$, is the number of edges in $G$ incident to $v$. 
Then $\sum_{v\in V} \deg(v) = 2|E|$.
Denoting the area of a region $R$ of $\mathbb{S}^2$ by $\area(R)$, we thus obtain
\begin{align*}
&\mathrel{\phantom{=}} \sum_{v\in V} \left(\frac{\deg(v)}{2}-3\right) + \sum_{f\in F}\left(\deg(f)-3 + \frac{3}{2\pi}\area(f)\right)\\
&= 3|E| - 3|V| - 3|F| + \frac{3}{2\pi}\sum_{f\in F}\area(f)\\
&\leq -6 + \frac{3}{2\pi}\area(\mathbb{S}^2) = 0.
\end{align*}
Thus, if we assign \emph{initial charges} of $\frac{\deg(v)}{2}-3$ to each vertex $v$ and $\deg(f)-3+\frac{3}{2\pi}\area(f)$ to each face $f$, then the total initial charge assigned to vertices and faces is $\leq 0$, with equality implying that $G$ is connected.

For each angle $\alpha$ made by an incident vertex-face pair (there is more than one such angle at a vertex $v$ if the removal of $v$ from $G$ increases the number of connected components), we now transfer a charge of
\[ c(\alpha) = \begin{cases}
    0 & \text{if $\alpha\leq \pi/3$}\\
    \frac{3}{2\pi}\alpha-\frac{1}{2} & \text{if $\pi/3\leq\alpha\leq 2\pi/3$}\\
    \frac{1}{2} & \text{if $\alpha\geq 2\pi/3$}
  \end{cases} \]
from the face to the vertex.
(This is the same discharging function as is used in \cite{kurz}.)
The charges after this transfer are called \emph{final charges}.
We next show that all final charges are non-negative.
Additionally, if the final charge of a vertex $v$ is $0$, then $v$ has to be of degree $5$, while if the final charge of a face $f$ is $0$, then the $f$ is either a triangle, quadrilateral or pentagon, and each interior angle of the face is in the interval $(\pi/3,2\pi/3]$.

If a vertex has degree $6$ or more, its initial charge is already non-negative.
If its final charge is $0$, then all its angles have to be $\leq\pi/3$.
Since the sum of its angles is $2\pi$, this means that each angle is equal to $\pi/3$.

Next consider a vertex of degree $5$ with angles $\alpha_1,\dots,\alpha_5$.
Its initial charge is $-1/2$, and $\sum_i\alpha_i=2\pi$.
If some $\alpha_i\geq 2\pi/3$, then a charge of $1/2$ is transferred, and the final charge is non-negative.
Otherwise, all $\alpha_i < 2\pi/3$, and the charge received by the vertex from angle $\alpha_i$ is $c(\alpha_i)\geq\frac{3}{2\pi}\alpha_i-\frac{1}{2}$.
Summing over the five angles, we obtain that the total charge received by the vertex is at least $1/2$,
and the final charge is also non-negative in this case.

We next show that all the faces have a non-negative final charge.
In total, there are exactly $\deg(f)$ angles at the vertices belonging to a face $f$.
The initial charge of $f$ is greater than $\deg(f)-3$, and a charge of at most $1/2$ is received from $f$ by each vertex belonging to it, which gives that the final charge of $f$ is greater than $\deg(f)-3-\deg(f)/2$, which is non-negative if $\deg(f)\geq 6$.
Note that this takes care of all faces with a boundary that is not a cycle.
Indeed, suppose that the boundary of a face $f$ is not a cycle.
If the boundary contains two cycles, then, since they both have length at least $3$ and are edge-disjoint, $\deg(f)\geq 6$.
If the boundary has no cycle, then it has a vertex of degree at most $1$, which contradicts the minimum degree assumption on $G$.
In the remaining case, the boundary of $f$ contains exactly one cycle, together with at least one more edge.
However, in this case we again obtain a vertex of degree $1$, a contradiction as before.

We still have to consider faces with at most $5$ boundary edges.
We have just proved that the boundaries of these faces are cycles.

Consider a triangular face $f$ with equal angles $\alpha > \pi/3$ and initial charge $\frac{3}{2\pi}\area(f)$.
Then $\area(f)=3\alpha-\pi$, and the total charge going to its vertices is \[3c(\alpha)\leq 3\left(\frac{3}{2\pi}\alpha-\frac{1}{2}\right)=\frac{3}{2\pi}\area(f).\]
Thus, the final charge of $f$ is non-negative, and equals $0$ only if $\alpha\leq 2\pi/3$.

Next consider a quadrilateral face $f$ with two opposite angles $\alpha$ and the other two opposite angles $\beta$, where $\alpha\leq\beta$.
Then $\area(f)=2\alpha+2\beta-2\pi$, which gives $\alpha+\beta > \pi$.
The initial charge is $\frac{3}{2\pi}\area(f)+1$.
If $\alpha\leq\pi/3$, then $\beta > 2\pi/3$, so the charge going to the vertices is $2c(\alpha)+2c(\beta)=2\cdot 0+2\cdot\frac{1}{2} < \frac{3}{2\pi}\area(f)+1$.
Otherwise, $\beta\geq\alpha > \pi/3$, and then the total charge going to the vertices is
\[2c(\alpha)+2c(\beta)\leq\frac{3}{2\pi}(2\alpha+2\beta)-\frac{4}{2} = \frac{3}{2\pi}\area(f)+1,\] and equality implies that $\pi/3 < \alpha\leq\beta\leq 2\pi/3$.

Finally, consider a pentagonal face $f$ with angles $\alpha_1,\dots,\alpha_5$.
Then $\area(f)=\sum_{i=1}^5\alpha_i-3\pi$, and the initial charge is $\frac{3}{2\pi}\area(f)+2$.
If some $\alpha_i\leq\pi/3$, then the total charge moving to the vertices is at most $4\cdot\frac{1}{2} < \frac{3}{2\pi}\area(f) + 2$.
Otherwise, all $\alpha_i > \pi/3$, and then the total charge transferred is
\[\sum_{i=1}^5c(\alpha_i)\leq\sum_{i=1}^5\left(\frac{3}{2\pi}\alpha_i-\frac{1}{2}\right) = \frac{3}{2\pi}\area(f)+2.\]
Equality implies that all $\alpha_i\in(\pi/3,2\pi/3]$.

Thus, all final charges are non-negative.
Since the total sum of the charges is non-positive, we get equality throughout, which implies that the graph is connected, all faces are triangles, quadrilaterals or pentagons with all their interior angles lying in the interval $(\pi/3,2\pi/3]$, and all vertices have degree~$5$.

To finish the proof, we show that $G$ is the contact graph of the packing of circular caps centered at the vertices of $G$, with each cap of angular diameter equal to the angular length $\lambda$ of the edges of $G$.
Note that it is sufficient to show that the angular distance between any two non-adjacent vertices is greater than $\lambda$.
We do this in two steps.
First, we show that the diagonals of the quadrilaterals and pentagons all have angular length $ > \lambda$.
Then we show that for any two vertices $x,y$ with $xy$ not an edge, its angular length $|xy| > \lambda$.
(Robinson \cite[p. 298]{Robinson1969} also considers the situation where the sphere is subdivided into convex polygons of side length $\lambda$ and all diagonals of length $ > \lambda$, and claims that this means that all distances between vertices are $\geq\lambda$.
Our second step proves this claim.)

\smallskip
\noindent\parbox{0.69\textwidth}{\quad First, consider a quadrilateral $abcd$.
Since all of its interior angles are $\leq 2\pi/3$, it contains its diagonals $ac$ and $bd$.
Since $\triangle bac$ and $\triangle dac$ are congruent isosceles spherical triangles, $\angle bac = \angle cad = \frac12\angle a \leq \pi/3$.
Since $\angle b > \pi/3$,
we have $\angle bac < \angle b$, hence the diagonal $|ac| > |bc| = \lambda$.
Similarly, $|bd| > |cd| = \lambda$.
}\hfill
\parbox{0.267\textwidth}{%
\begin{tikzpicture}[thick, scale=0.8]
\small
    \coordinate [label=above:$a$] (a) at (0,0);
    \coordinate [label={[label distance = 2mm]left:$b$}] (b) at (2,1.25);
    \coordinate [label=above:$c$] (c) at (4,0);
    \coordinate [label={[label distance = 2.2mm]left:$d$}] (d) at (2,-1.25);
    \draw (a) -- (b) 
              -- node[midway, above right=-0.5mm] {$\lambda$} (c)
              -- node[midway, below right=-0.7mm] {$\lambda$} (d) -- cycle;
    \draw[dashed] (a) -- (c); 
    \foreach \x in {(a),(b),(c),(d)} \draw[fill] \x circle (1.5pt);
    \draw (a) + (0.6,0.15) node {\scriptsize$\circ$};
    \draw (a) + (0.6,-0.15) node {\scriptsize$\circ$};
    \draw (c) + (-0.6,0.15) node {\scriptsize$\circ$};
    \draw (c) + (-0.6,-0.15) node {\scriptsize$\circ$};
\end{tikzpicture}
}

\medskip
\noindent\parbox{0.69\textwidth}{\quad Next, consider a pentagon $abcde$.
Again, because the interior angles are all at most $2\pi/3$, all its diagonals are in its interior.
Let $be$ be the shortest of the five diagonals, and suppose $|be| \leq\lambda$.
Then $\angle abe = \angle aeb \geq \angle a > \pi/3$.
Since $\angle b, \angle e \leq 2\pi/3$, it follows that $\angle deb, \angle bed < \pi/3$.
Since the sum of the angles of $\triangle bed$ is more than $\pi$, it follows that $\angle bde > \pi/3 > \angle bed$,
hence $|be| > |bd|$, a contradiction.
Thus, $|be| > \lambda$, hence all diagonals are longer than $\lambda$.
\\[-2.5mm]
}\hfill
\parbox{0.28\textwidth}{%
\begin{tikzpicture}[thick, scale=0.89]
\small
    \coordinate [label=left:$b$] (b) at (-0.3,1.4);
    \coordinate [label=below:$a$
                ] (a) at (1.5,2.5);
    \coordinate [label=right:$e$] (e) at (3.3,1.4);
    \coordinate [label=right:$d$] (d) at (2.7,-0.4);
    \coordinate [label=left:$c$] (c) at (0.3,-0.4);
    \draw (a) -- (b) node[midway,above left] {$\lambda$} -- (c) -- (d) -- (e) -- cycle node[midway,above right] {$\lambda$};
    \draw[dashed] (b) -- node[midway,above=-0.5mm] {$\leq\lambda$}
                (e);
    \draw[dashed] (b) -- (d);
    \foreach \x in {(a),(b),(c),(d),(e)} \draw[fill] \x circle (1.5pt);
    \draw (b) + (0.4,0.12) node {\scriptsize$\circ$};
    \draw (b) + (0.8,-0.2); 
    \draw (e) + (-0.4,0.12) node {\scriptsize$\circ$};
    \draw (d) + (-0.1,0.35); 
    \draw (e) + (-0.3,-0.2); 
\end{tikzpicture}}

\medskip
\noindent\parbox{0.69\textwidth}{\quad Finally, suppose that there are two vertices $x,y$ that do not form an edge of $G$, but $|xy|\leq\lambda$.
We may take $xy$ to be the shortest non-edge.
We have already seen that $x$ and $y$ cannot belong to the same face.
Then the arc $xy$ intersects some edge $zw$ of $G$ at some point $p\neq x,y,z,w$.
We may suppose that no edge of $G$ intersects the arc $xp$ at some point strictly between $x$ and $p$, otherwise we replace $zw$ with that edge.
Note that $z$ and $x$ are vertices of the face \\[-2.8mm]
}\hfill
\parbox{0.28\textwidth}{%
\begin{tikzpicture}[thick, scale=0.85]
\small
    \coordinate [label=left:$x$] (x) at (2,4);
    \coordinate [label=left:$y$] (y) at (2,1);
    \coordinate (z) at (0,2);
    \coordinate (w) at (4,2);
    \coordinate (p) at (2,2);
    \draw (3,3) node{$f$};
        \draw (p) node[above right] {$p$};
        \draw (z) node[left] {$z$};
        \draw (w) node[right] {$w$};
        \draw (z) -- (w);
    \draw[dashed] (x) -- (y);
    \foreach \x in {(x),(y),(z),(w)} \draw[fill] \x circle (1.5pt);
    \draw (z) -- (w);
    \draw[dashed] (x)--(z);
    \draw[dashed] (y) -- (w);
\end{tikzpicture}}

\noindent 
$f$ that contains the arc $xp$, since we can move from $x$ to $z$ first along the arc $xp$, and then right next to the arc $pz$, without touching $\|G\|$ except at $x$ and $z$.
Then $|xz|\geq \lambda$, since $xz$ is either an edge or a diagonal of $f$.
By the triangle inequality applied to $\triangle pxz$ and $\triangle pyw$,
\[|xy|+\lambda = |xy| + |zw| > |xz| + |wy| \geq \lambda + |wy|,\] hence $|wy| < |xy|$.
This implies that $wy$ is a non-edge of $G$ shorter than $xy$, contradicting the minimality of $xy$.

We have shown that all non-edge distances are strictly greater than $\lambda$.
Therefore, $G$ is the contact graph of the packing of circular caps of angular diameter $\lambda$, centered at the vertices of $G$, as desired.
\end{proof}

\section{Final remarks}
\begin{enumerate}[wide, labelindent=0pt]
\item In the above proof we did not assume that $G$ is connected.
Instead, connectedness of $G$ was a natural consequence of the proof.
(Alternatively, if we assumed $G$ to be connected, then the Euler formula would have been an equality, and at the end of the proof we would have to observe that none of Robinson's five contact graphs leave any space for more than one of them occurring on the same sphere.)

\item There are many $3$- and $4$-regular matchstick graphs on the sphere, with the smallest ones coming from the regular tetrahedron and the regular octahedron.
The Archimedean prisms and antiprisms are two infinite families of polyhedra with regular faces inscribed in a sphere, giving infinite families of $3$- and $4$-regular contact graphs of packings of spherical caps on the sphere.
Robinson \cite{Robinson1969} provides many further examples of arbitrarily large $3$- and $4$-regular contact graphs.
Since a planar graph with $n$ vertices has at most $3n-6$ edges, there are of course no $k$-regular planar graphs for any $k\geq 6$.

\item As a consequence of our main theorem we obtain a characterization of all $5$-regular matchstick graphs in the elliptic plane.
The \emph{elliptic plane} is the projective plane together with the angular metric inherited from the sphere by identifying antipodal pairs of points.
As Robinson observed, exactly one of the five $5$-regular packings of circular caps is centrally symmetric, and it follows that there is exactly one $5$-regular matchstick graph in the elliptic plane, namely the complete graph on $6$ vertices coming from the well-known configuration of $6$ equidistant points \cite{Haantjes}.

\item The proof of Kurz and Pinchasi \cite{kurz} also goes through for the hyperbolic plane, giving that there are no $5$-regular matchstick graphs of any hyperbolic length.

\item The proof in \cite{kurz} can be combined with the isoperimetric inequality to show that a matchstick graph in the Euclidean plane is far from $5$-regular: In any matchstick graph with $n$ vertices, the number of vertices of degree at most $4$ is at least $c\sqrt{n}$ \cite{LS2}.
\end{enumerate}

\section*{Acknowledgment}
We thank the anonymous reviewers for their helpful suggestions.

\end{document}